\documentclass[12pt]{amsart}
\usepackage{amssymb,latexsym, amsthm,epsf}

\newtheorem{thm}{Theorem}[section] 
\newtheorem{defn}[thm]{Definition}
\newtheorem{obs}[thm]{Observation}

\newtheorem{cor}[thm]{Corollary}

\newtheorem{lem}[thm]{Lemma}

\newtheorem{prop}[thm]{Proposition}
\newtheorem{rem}[thm]{Remark}

\def\cL{{\mathcal L}}
\def\C{{\mathbb C}}

\def\R{{\mathbb R}}
\def\P{{\mathbb P}}
\def\PP{{\mathbb P}}

\def\({\left(}
\def\){\right)}

\long\def\forget#1\forgotten{}

\newif \iffurther 
\furtherfalse

\newif \iffurther 
\furtherfalse

\newif\ifXY 
\XYfalse    

\ifXY
\usepackage{xy}
\fi
\ifXY
\xyoption{all}
\fi

\begin{document}

\title[Conjugation-free geometric groups of arrangements II]{A conjugation-free geometric presentation of fundamental groups of arrangements II: Expansion and some properties}

\author[Eliyahu, Garber, Teicher]{Meital Eliyahu$^1$, David Garber and Mina Teicher}

\stepcounter{footnote}
\footnotetext{Partially supported by the Israeli Ministry of Science and Technology.}

\address{Meital Eliyahu, Department of Mathematics, Bar-Ilan University, 52900 Ramat-Gan, Israel}
\email{eliyahm@macs.biu.ac.il}

\address{David Garber, Department of Applied Mathematics, Faculty of
  Sciences, Holon Institute of Technology, 52 Golomb st., PO
  Box 305, 58102 Holon, Israel}
\email{garber@hit.ac.il}

\address{Mina Teicher, Department of Mathematics, Bar-Ilan University, 52900 Ramat-Gan, Israel}
\email{teicher@macs.biu.ac.il}

\begin{abstract}
A {\em conjugation-free geometric presentation} of a fundamental group
is a presentation with the natural topological generators $x_1, \dots, x_n$
and the cyclic relations:
$$x_{i_k}x_{i_{k-1}} \cdots x_{i_1} = x_{i_{k-1}} \cdots x_{i_1} x_{i_k} = \cdots = x_{i_1} x_{i_k} \cdots x_{i_2}$$
with no conjugations on the generators.

We have already proved in \cite{EGT} that if the graph of the arrangement is a disjoint union of cycles, then its fundamental group has a conjugation-free geometric presentation.
In this paper, we extend this property to arrangements whose graphs are a disjoint union of cycle-tree graphs.

Moreover, we study some properties of this type of presentations for a fundamental group of a line arrangement's complement. We show that these presentations satisfy a completeness property in the sense of Dehornoy, if the corresponding graph  of the arrangement has no edges. The completeness property is a powerful property which leads to many nice properties concerning the presentation (such as the left-cancellativity of the associated monoid and yields some simple criterion for the solvability of the word problem in the group).
\end{abstract}


\maketitle

\section{Introduction}
The fundamental group of the complement of plane curves is a very important topological invariant, which can be also computed for line arrangements. We present here some applications of this invariant.

Chisini \cite{chisini}, Kulikov \cite{Kul,Kul2} and Kulikov-Teicher \cite{KuTe} have used the fundamental group of complements of branch curves of
generic projections in order to distinguish between connected components of the moduli space of smooth projective surfaces, see also \cite{FrTe}.

Moreover, the Zariski-Lefschetz hyperplane section theorem (see \cite{milnor})
states that:
$$\pi_1 (\C\P^N \setminus S) \cong \pi_1 (H \setminus (H \cap S)),$$
where $S$ is a hypersurface and $H$ is a generic 2-plane.
Since $H \cap S$ is a plane curve, the fundamental groups of complements of curves
can be used also for computing the fundamental groups of complements of hypersurfaces in $\C\P^N$.

A different need for computations of fundamental groups is to
obtain more examples of Zariski pairs \cite{Z1,Z2}. A pair of plane
curves is called {\it a Zariski pair} if they have the
same combinatorics (to be exact: there is a degree-preserving bijection between the set of irreducible
components of the two curves $C_1,C_2$, and there exist regular neighbourhoods of the curves $T(C_1),T(C_2)$
such that the pairs $(T(C_1),C_1),(T(C_2),C_2)$ are homeomorphic and the homeomorphism respects the bijection
above \cite{AB-CR}), but their complements in $\P^2$ are not homeomorphic.
For a survey, see \cite{ABCT}.

It is also interesting to explore new finite non-abelian groups which
serve as fundamental groups of complements of plane curves in general;
see for example \cite{AB,AB1,Deg,Z1}.

\medskip

An {\it affine line arrangement} in $\C^2$ is a union of copies of $\C^1$ in $\C^2$. Such an arrangement is called {\em real} if the defining equations of all its lines can be written with real coefficients, and {\em complex} otherwise. Note that the intersection of a real arrangement with the natural copy of $\R^2$ in $\C^2$ is an arrangement of lines in the real plane, called the {\it real part} of the arrangement.

Similarly, a {\it projective line arrangement} in $\C\PP^2$ is a union of copies of $\C\PP^1$ in $\C\PP^2$. Note that the realization of the MacLane configuration \cite{maclane} is an example of a complex arrangement; see also \cite{ABC,Ry}.

For real and complex line arrangements $\mathcal L$, Fan \cite{Fa2} defined a graph $G(\mathcal L)$ which is associated to its multiple points (i.e., points where more than two lines are intersected). We give here its version for real arrangements: Given a real line arrangement $\mathcal L$, the graph $G(\mathcal L)$ of multiple points lies on the real part of $\mathcal L$. It consists of the multiple points of $\mathcal L$, with the
segments between the multiple points on lines which have at least two multiple points. Note that if the arrangement consists of three
multiple points on the same line, then $G(\mathcal L)$ has three vertices on the same line (see Figure \ref{graph_GL}(a)).
If two such lines happen to intersect in a simple point (i.e., a point where exactly two lines are intersected), it is ignored
(i.e., the lines do not meet in the graph). See another example in Figure \ref{graph_GL}(b) (note that Fan's definition gives a graph different from the graph defined in \cite{JY,WY}).

\begin{figure}[!ht]
\epsfysize 4cm
\centerline{\epsfbox{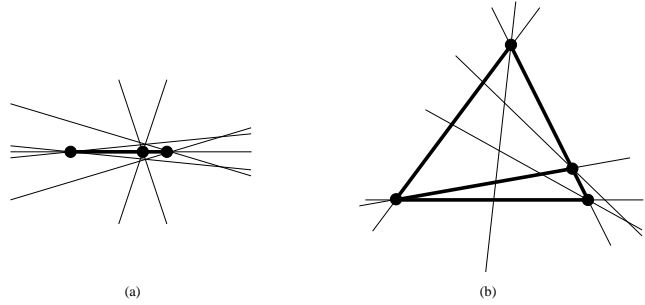}}
\caption{Examples for $G(\mathcal L)$}\label{graph_GL}
\end{figure}

\medskip

In \cite{EGT} we introduce the notion of a {\it conjugation-free geometric presentation} of the fundamental group of an arrangement:
\begin{defn}
Let $G$ be a fundamental group of the affine or projective complements of some line arrangement with $n$ lines. We say that $G$ has {\em a conjugation-free geometric presentation} if $G$ has a presentation with the following properties:
\begin{itemize}
\item In the affine case, the generators $\{ x_1,\dots, x_n \}$ are the meridians of lines at some far side of the arrangement, and therefore the number of generators is equal to $n$.
\item In the projective case, the generators are the meridians of lines at some far side of the arrangement except for one, and therefore the number of generators is equal to  $n - 1$.
\item In both cases, the relations are of the following type:
$$x_{i_k} x_{i_{k-1}} \cdots x_{i_1} = x_{i_{k-1}} \cdots x_{i_1} x_{i_k} = \cdots = x_{i_1} x_{i_k} \cdots x_{i_2},$$
where $\{ i_1,i_2, \dots , i_k \} \subseteq \{1, \dots, m \}$ is an increasing subsequence of indices,
where $m=n$ in the affine case and $m=n-1$ in the projective case. Note that for $k=2$, we get the usual commutator.
\end{itemize}

\end{defn}

Note that in usual geometric presentations of the fundamental group, most of the relations have conjugations.

\medskip

The importance of this family of arrangements is that the fundamental group can be read directly from the arrangement or equivalently from its incidence lattice (where the {\em incidence lattice} of an arrangement is the partially-ordered set of non-empty intersections of the lines, ordered by
inclusion, see \cite{OT}) without any computation. Hence, for this family of arrangements, the incidence lattice determines the fundamental group of the complement (this is based on Cordovil \cite{cor}, too).

\medskip

We start with the easy fact that there exist arrangements whose fundamental groups have no conjugation-free geometric presentation: The fundamental group of the Ceva arrangement (also known as the {\it braid arrangement}, appears in Figure \ref{ceva}) has no conjugation-free geometric presentation (see \cite{EGT}).

\begin{figure}[!ht]
\epsfysize 4cm
\centerline{\epsfbox{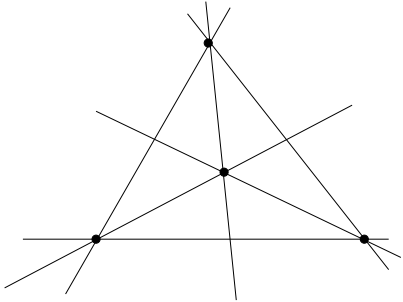}}
\caption{Ceva arrangement}\label{ceva}
\end{figure}

Note also that if the fundamental groups of two arrangements $\cL_1,\cL_2$  have conjugation-free geometric presentations and the arrangements intersect transversally, then the fundamental group of $\cL_1 \cup \cL_2$ has a conjugation-free geometric presentation, too. This is due to the important result of Oka and Sakamoto \cite{OkSa}:
\begin{thm}\label{OS_thm} {\bf (Oka-Sakamoto)} Let $C_1$ and $C_2$ be algebraic plane curves in $\C ^2$.
Assume that the intersection $C_1 \cap C_2$
consists of distinct $d_1 \cdot  d_2$ points, where $d_i \ (i=1,2)$ are the
respective degrees of $C_1$ and $C_2$. Then:
$$\pi _1 (\C ^2 - (C_1 \cup C_2)) \cong \pi _1 (\C ^2 -C_1) \oplus \pi _1 (\C ^2 -C_2)$$
\end{thm}

The main result of \cite{EGT} is:
\begin{prop}
The fundamental groups of following family of arrangements have a conjugation-free geometric presentation: a real arrangement $\mathcal L$, where $G(\mathcal L)$ is a disjoint union of cycles of any length, and the multiplicities of the multiple points are arbitrary.
\end{prop}

\medskip

In this paper, we continue the investigation of the family of arrangements whose fundamental groups have conjugation-free geometric presentations in two directions.
First, we extend this property to real arrangements whose graphs are a disjoint union of cycle-tree graphs, where an example for a cycle-tree graph is presented in Figure \ref{cycle-tree-intro}
(see Definition \ref{def-cycle-tree} below).

\begin{figure}[!ht]
\epsfysize 3cm
\centerline{\epsfbox{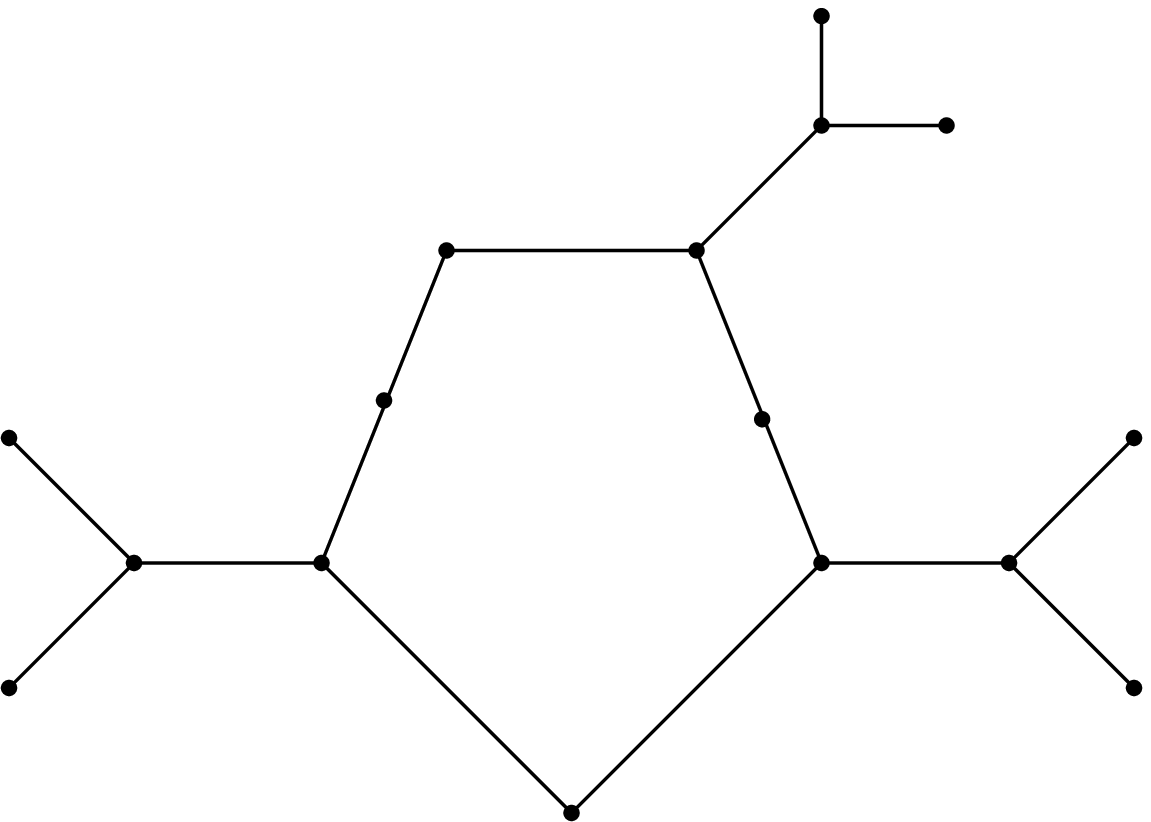}}
\caption{An example of a cycle-tree graph}\label{cycle-tree-intro}
\end{figure}

In the second direction, we study some properties of this type of presentations for a fundamental group of a line arrangement's complement. We prove:
\begin{prop}\label{prop-complete}
Let $\cL$ be a real arrangement whose fundamental group has a conjugation-free geometric presentation and its graph $G(\cL)$ has no edges. Then the presentation of the corresponding monoid is complete (and complemented).
\end{prop}

The completeness property is a powerful property which leads to many nice properties concerning the presentation (such as the left-\break cancellativity of the associated monoid and yields some simple criterion for the solvability of the word problem in the group and for Garside groups).

\medskip

The paper is organized as follows. In Section \ref{line_addition}, we prove that arrangements whose graphs are a disjoint union of cycle-tree graphs have a conjugation-free geometric presentation of the fundamental group of the complement. In Section \ref{sec-complemented}, we prove that conjugation-free geometric presentations are complemented presentations. Section \ref{sec-complete} deals with complete presentations, and includes the proof of Proposition \ref{prop-complete}.

\section{Adding a line through a single point preserves the conjugation-free geometric presentation}\label{line_addition}

We start with the following obvious observation, which is based on the Oka-Sakamoto decomposition theorem (see Theorem \ref{OS_thm} above):
\begin{obs}\label{obs_line_add}
Let $\cL$ be an arrangement whose fundamental group has a conjugation-free geometric presentation.
Let $L$ be a line which intersects $\cL$ transversally. Then
$\cL \cup L$ is also an arrangement whose fundamental group has a conjugation-free geometric presentation.
\end{obs}

In this section, we prove the following proposition, which is the next step:

\begin{prop}\label{prop_line_add}
Let $\cL$ be a real arrangement whose affine fundamental group has a conjugation-free geometric presentation.
Let $L$ be a line not in $\cL$, which passes through one intersection point $P$ of $\cL$. Then
$\cL \cup L$ is also an arrangement whose affine fundamental group has a conjugation-free geometric presentation.
\end{prop}

\begin{proof}
We can assume that the point $P$ is the leftmost and lowest point of the arrangement $\cL$ and all the intersection points of the line $L$ (except for $P$) are to the left of all the intersection points of the arrangement $\cL$ (except for $P$). We can also assume that the highest line in $\cL$ (with respect to the global numeration of the lines) passes through $P$. See Figure \ref{fig_L_union_l} for an illustration, where the arrangement $\cL$ is in the dashed rectangle.

\begin{figure}[!ht]
\epsfysize 4cm
\centerline{\epsfbox{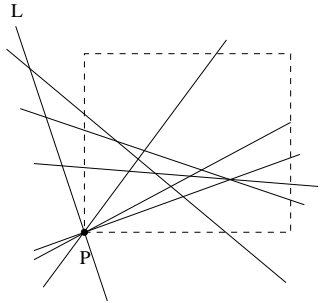}}
\caption{An illustration of the real part of $\cL \cup L$}\label{fig_L_union_l}
\end{figure}

The above assumption is due to the following: First, one can rotate a line that participates in only one multiple point as long as it does not unite with a different line (by Results 4.8 and 4.13 of \cite{GTV}). Second, moving a line that participates in only one multiple point over a different line (see Figure \ref{triangle-line}) is permitted in the case of a triangle due to a result of Fan \cite{Fa2} that the family of configurations with $6$ lines and three triple points is connected by a finite sequence of smooth equisingular deformations. Moreover, by Theorem
4.11 of \cite{GTV}, one can assume that the point $P$ is the leftmost point of the arrangement $\cL$.

\begin{figure}[!ht]
\epsfysize 3cm
\centerline{\epsfbox{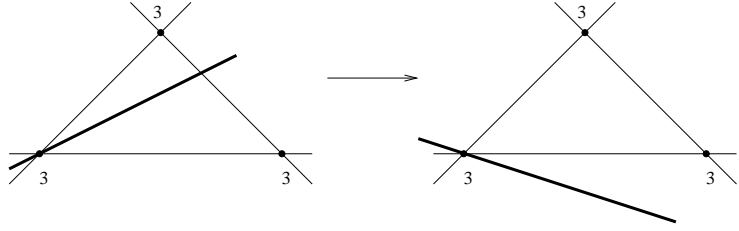}}
\caption{Moving a line that participates in only one multiple point over a different line}\label{triangle-line}
\end{figure}

Let $n$ be the number of lines in $\cL$ and let $m$ be the multiplicity of $P$ in $\cL$.
So, the list of Lefschetz pairs of the arrangement $\cL$ is $$([a_1,b_1],[a_2,b_2],\dots,[a_{q-1},b_{q-1}],[1,m]),$$
where the Lefschetz pair $[1,m]$ corresponds to the point $P$ (for the theory used here for computing the fundamental group of the complements of arrangements, see \cite{EGT,GaTe,vK,BGT1}).
Since we have that $\pi_1(\C^2-\cL)$ has a conjugation-free geometric presentation, then we know that all the conjugations in the relations
induced by the van Kampen Theorem \cite{vK} (see also \cite{EGT}) can be simplified.

Now, let us deal with the arrangement $\cL \cup L$. By our assumptions, its list of Lefschetz pairs is (we write in small brackets the name of the corresponding point):
$$([a_1+1,b_1+1]_{(p_1)},[a_2+1,b_2+1]_{(p_2)},\dots,[a_{q-1}+1,b_{q-1}+1]_{(p_{q-1})},$$
$$[1,m+1]_{(p_q)},[m+1,m+2]_{(p_{q+1})},[m+2,m+3]_{(p_{q+2})},\dots, [n,n+1]_{(p_{q+(n-m)})} ).$$

\medskip

We start with the relations induced from intersection points on the line $L$. We first choose a set of $n+1$ generators of the fundamental group of its complement corresponding to its lines, namely $\{x_1,\dots,x_{n+1}\}$.
By the Moishezon-Teicher algorithm \cite{GaTe,BGT1} (see also \cite{EGT}), we now compute the skeletons corresponding to the points on the line $L$ (i.e. the points $p_j$, where $q \leq j \leq q+n-m$). Note that:
\begin{tiny}$$ \Delta \langle a_1+1,b_1+1 \rangle \Delta \langle a_2+1,b_2+1 \rangle \cdots \Delta \langle a_{q-1}+1,b_{q-1}+1 \rangle = \Delta \langle 2,n+1 \rangle \Delta ^{-1} \langle 2,m+1 \rangle,$$\end{tiny}
\hspace{-5pt}since given an arrangement, the multiplication of all the halftwists based on its Lefschetz pairs is equivalent to a unique halftwist of all the lines. By this observation, we get the skeletons in Figure \ref{bm_L_union_l}.

\begin{figure}[!ht]
\epsfysize 2cm
\centerline{\epsfbox{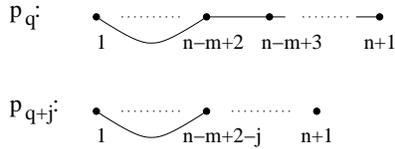}}
\caption{The skeletons of the points $p_{q+j}$ where\break $0 \leq j \leq n-m$}\label{bm_L_union_l}
\end{figure}

Hence, we get the following relations:

\medskip

\noindent
For the point $p_q$:
\begin{eqnarray*}
x_{n+1} x_n\cdots x_{n-m+2} x_1 & = & x_1 x_{n+1} x_n\cdots x_{n-m+2}=\\
&=& \cdots = x_n\cdots x_{n-m+2} x_1 x_{n+1}
\end{eqnarray*}
For the points $p_j$, where $q+1 \leq j \leq q+n-m$:
$$x_{n-m+2-j} x_1=  x_1 x_{n-m+2-j}$$
These relations are obviously without conjugations.

\medskip

Now, we move to the relations induced from points appearing in the original arrangement $\cL$.
The only change in the level of the Lefschetz pairs is an addition of one index in all the pairs, due to the line $L$.
Therefore, the induced braid monodromy and the relations will be changed by adding $1$ to every index, i.e.,
if we have a relation which involves the generators $x_{i_1},\dots, x_{i_k}$, then after adding the line, we have the same relation but with
generators $x_{i_1+1},\dots, x_{i_k+1}$, respectively.

Now, we know that the fundamental group of $\cL$ has a conjugation-free geometric presentation;
hence we have that by a simplification process, one can reach a presentation without conjugations.
If we imitate the simplification process of the presentation of the fundamental group of $\cL$ for the presentation of the fundamental group of $\cL \cup L$,
the cases in which we need to use the relations induced from the point $P$ are the relations that have been simplified by using the relations induced from $P$ before adding the line.
As above,  the  original relations induced from $P$ are:
\begin{eqnarray*}
R_p: \quad x_n x_{n-1}\cdots x_{n-m+1} &=& x_{n-1}\cdots x_{n-m+1} x_n =\\
& = & \cdots = x_{n-m+1} x_n \cdots x_{n-m+2},
\end{eqnarray*}
while the new ones are:
\begin{eqnarray*}
\tilde{R}_p: \quad x_{n+1} x_n\cdots x_{n-m+2} x_1 & = & x_1 x_{n+1} x_n\cdots x_{n-m+2}=\\
& = & \cdots = x_n\cdots x_{n-m+2} x_1 x_{n+1}.
\end{eqnarray*}

We can divide the relations induced from $\cL$ before adding the line $L$ into two subsets:
\begin{enumerate}
\item Relations that during the simplification process contain the subword $x_{n-m+2}^{-1} \cdots x_{n-1}^{-1} x_n x_{n-1} \cdots x_{n-m+2}$.
\item Relations that do not contain the above subword during its simplification process.
\end{enumerate}

For the second subset, the simplification process will be identical before adding the line $L$ and after it, since all the other relations induced by $\cL$ have not been changed by adding the line $L$ (except for adding 1 to the indices).

For the first subset, let us denote the relation by $R$.  Except for applying the relations induced from $P$, the rest of simplification process is identical to the one before adding the line (again, except for adding 1 to the indices). The only change is in the step of applying $R_p$. In this step, before adding the line $L$, the generator $x_1$ has not been involved in $R_p$, but after adding the line $L$, it appears in $\tilde{R}_p$. Hence, for applying $\tilde{R}_p$, we have to conjugate relation $R$ by $x_1$, and using the commutative relations which $x_1$ is involved in, we can diffuse $x_1$ into the relation $R$, so we can use the relation $\tilde{R}_p$ instead of $R_p$.

Hence, we can simplify all the conjugations in all the relations, so we have a conjugation-free geometric presentation, as needed.
\end{proof}

\begin{rem}
Note that adding a line which closes a cycle in $\cL$ might not preserve the conjugation-free geometric presentation property. For example, adding a line to an arrangement of 5 lines which creates the Ceva arrangement (see Figure \ref{ceva}) is not an action which preserves the conjugation-free geometric presentation property.
\end{rem}

Hence, we can extend the family of arrangements whose fundamental groups have a conjugation-free geometric presentation. We start with the following definition:

\begin{defn}\label{def-cycle-tree}
A {\em cycle-tree graph} is a graph which consists of a cycle, where each vertex of the cycle can be a root  of a tree, see Figure \ref{cycle-tree}.
It is possible that there exist some vertices also in the middle of an edge of the cycle or the trees.

\begin{figure}[!ht]
\epsfysize 3cm
\centerline{\epsfbox{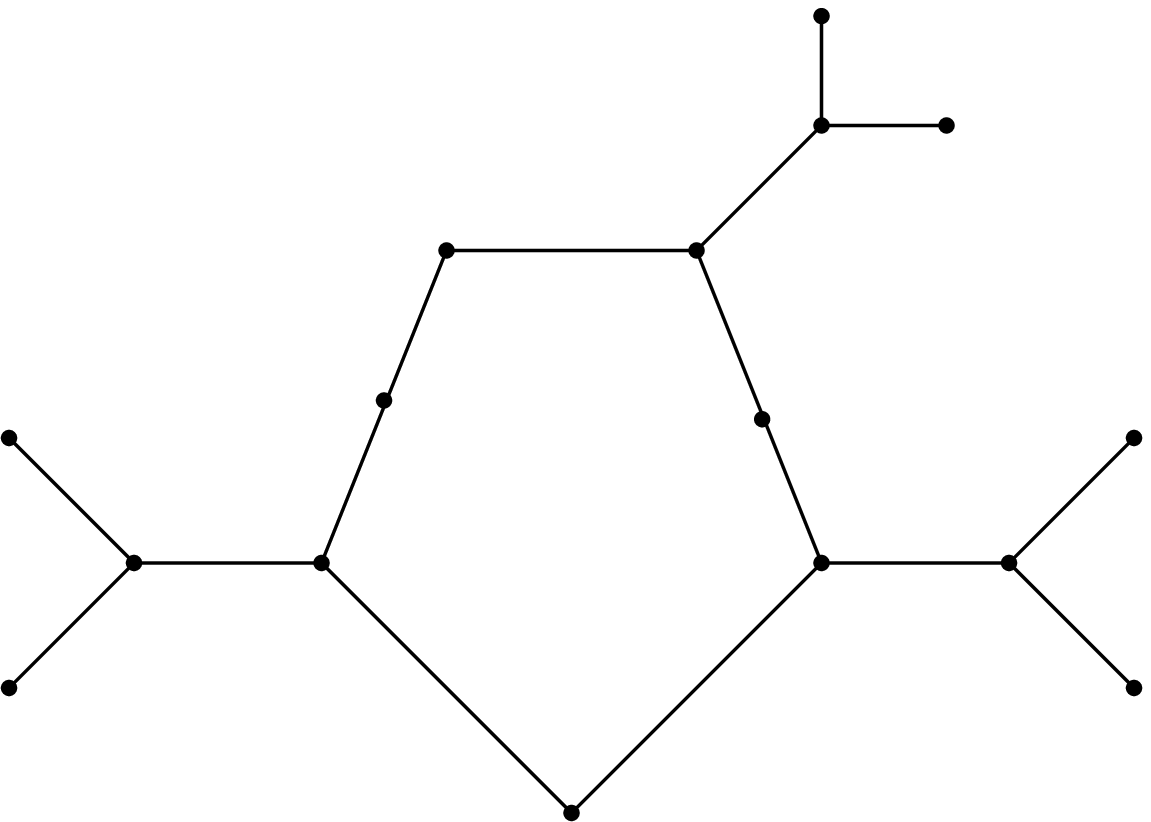}}
\caption{An example of a cycle-tree graph}\label{cycle-tree}
\end{figure}

\end{defn}

\begin{cor}
Let $\cL$ be a real line arrangement whose graph is a disjoint union of cycle-tree graphs. Then the fundamental group of $\cL$ has a conjugation-free geometric presentation.
\end{cor}

\begin{proof}
We start by proving that a real arrangement whose graph is a cycle-tree graph has a fundamental group which has a conjugation-free geometric presentation.
We already have from \cite{EGT} that a real arrangement whose graph is a cycle has a fundamental group which has a conjugation-free geometric presentation. By Proposition \ref{prop_line_add}, adding a line which is either transversal to an arrangement or passes through one intersection point, preserves the property that the fundamental group has a conjugation-free geometric presentation. One can easily construct an arrangement whose graph is a cycle-tree graph from an arrangement whose graph is a cycle by inductively adding a line which is either transversal to the arrangement or passes through one of its intersection points. Hence, we get that an arrangement whose graph is a cycle-tree graph has a fundamental group which has a conjugation-free geometric presentation.

In the next step, using the theorem of Oka and Sakamoto \cite{OkSa} (see Theorem \ref{OS_thm} above), we can generalize the result from the case of one cycle-tree graph to the case of a disjoint union of cycle-tree graphs.
\end{proof}


\section{Complemented presentations}\label{sec-complemented}

A {\em semigroup presentation} $(\mathcal S,\mathcal R)$ consists of a nonempty set $\mathcal S$ and a family of pairs of nonempty words
$\mathcal R$ in the alphabet $\mathcal S$. The corresponding monoid $(\mathcal S,\mathcal R)$ is
$\langle \mathcal S | \mathcal R \rangle ^+ \cong (\mathcal S^*/\equiv _{\mathcal R}^+)$.

Dehornoy \cite{Deh1} has defined the notion of a {\it complemented presentation} of a semigroup:

\begin{defn}
A semigroup presentation $(\mathcal S,\mathcal R)$ is called
{\em complemented} if, for each $s \in \mathcal S$, there is no relation $s \ldots = s \ldots$ in $\mathcal R$ and, for $s, s' \in \mathcal S$,
there is at most one relation $s \ldots = s' \ldots$ in $\mathcal R$.
\end{defn}

Our type of presentations satisfies this property as follows:
\begin{lem}
 A conjugation-free geometric presentation is a complemented presentation.
\end{lem}

\begin{proof}
Any pair of lines intersect exactly once, hence their corresponding generators appear as prefixes in exactly one relation.
Since there are no conjugations, this is their unique appearance as a pair of prefixes.
\end{proof}

\begin{rem} \ \\
\begin{enumerate}
\item This property does not hold for presentations of fundamental groups in general (due to the conjugations in the relations).
\item This property does not hold in the homogeneous minimal presentations introduced by Yoshinaga \cite{Yos}.
\end{enumerate}
\end{rem}

\section{Complete presentations}\label{sec-complete}

In this section, we will study which cases of conjugation-free geometric presentations are also {\em complete} in the sense of Dehornoy \cite{Deh2}.
The completeness property is a very important and powerful property.  In Section \ref{complete_back},
we supply some background on this property and some of its applications. In Section \ref{result}, we present our results in this direction.

\subsection{Background on complete presentations}\label{complete_back}

We follow the survey of Dehornoy \cite{Deh3}.
We start by defining the notion of a {\it word reversing}:

\begin{defn}
For a semigroup presentation $(\mathcal S,\mathcal R)$ and\break $w,w' \in ({\mathcal S} \cup {\mathcal S}^{-1})^*$,
{\em $w$ reverses to $w'$ in one step}, denoted by $w \curvearrowright _{\mathcal R} ^1 w'$,
if there exist a relation $sv' = s'v$ of $\mathcal R$ and $u,u'$ satisfying:
$$ w = us^{-1}s'u' \ {\rm and} \  w'= uv'v^{-1}u'.$$
We say that {\em $w$ reverses to $w'$ in $k$ steps}, denoted by $w \curvearrowright _R ^k w'$, if there exist words $w_0, \dots ,w_k$
satisfying $w_0 = w, w_k = w'$ and $w_i \curvearrowright _R ^1 w_{i+1}$ for each $i$.
The sequence $(w_0, \dots ,w_k)$ is called {\em an $R$-reversing sequence from $w$ to $w'$}.

We write $w \curvearrowright w'$, if $w \curvearrowright _{\mathcal R} ^k w'$ holds for some $k \in \mathbb{N}$.
\end{defn}

\begin{defn}
A semigroup presentation $(\mathcal S,\mathcal R)$ is called {\em complete} if, for all words $w,w' \in \mathcal S^*$:
$$w \equiv ^+_{\mathcal R} w'  \quad \Rightarrow \quad  w^{-1}w' \curvearrowright _{\mathcal R} \varepsilon .$$
where $\varepsilon$ is the empty word.
\end{defn}

In the next definition, we define the {\it cube condition}, which is a useful tool for verifing the completeness property.

\begin{defn}
Let $(\mathcal S, \mathcal R)$ be a semigroup presentation, and\break $u, u', u'' \in \mathcal S^*$. We say that $(\mathcal S, \mathcal R)$ satisfies the {\em cube condition for $(u, u', u'')$} if:
$$u^{-1} u'' u''^{-1} u' \curvearrowright_{\mathcal R} v' v^{-1} \qquad \Rightarrow \qquad
(u v')^{-1} (v u') \curvearrowright_{\mathcal R} \varepsilon .$$

For $X \subseteq \mathcal S^*$, we say that $(\mathcal S, \mathcal R)$ satisfies {\em the cube condition on $X$} if it satisfies the cube condition for every triple $(u, u', u'')$ where $u, u', u'' \in X$.
\end{defn}

\begin{figure}[!ht]
\epsfysize 3cm
\centerline{\epsfbox{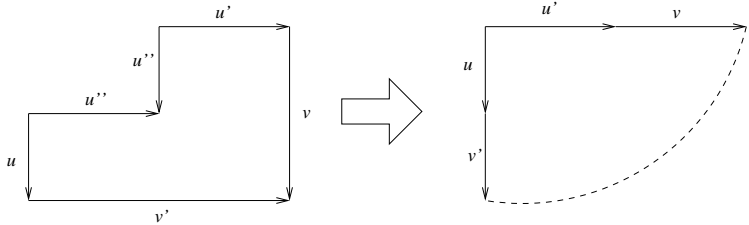}}
\caption{An illustration of the cube condition}\label{cube}
\end{figure}

\begin{defn}
A semigroup presentation $(\mathcal S, \mathcal R)$ is said to be {\em homogeneous} if there exists an $\equiv^+_{\mathcal R}$-invariant mapping $\lambda: \mathcal S^* \to \mathbb{N}$
satisfying, for $s \in \mathcal S$ and $w \in \mathcal S^*$, $$\lambda(s w) > \lambda(w).$$
\end{defn}

A typical case of a homogeneous presentation is where all relations in $\mathcal R$ preserve the length of words, i.e., they have the form
$v'=v$ where $v'$ and $v$ have the same length.

Dehornoy \cite{Deh2} has proved the following result:
\begin{prop}
Assume that $(\mathcal S, \mathcal R)$ is a homogeneous semigroup presentation.  Then
$(\mathcal S, \mathcal R)$ is complete if and only if it satisfies the cube condition on $\mathcal S$.
\end{prop}

The next definition is needed for introducing an operation used in an equivalent condition for the cube condition:

\begin{defn}
For a complemented semigroup presentation $(\mathcal S, \mathcal R)$ and $w, w' \in \mathcal S^*$, the {\em $\mathcal R$-complement of $w'$ in $w$, denoted $w \backslash w'$}, (``$w$ under $w'$''), is the unique word $v' \in \mathcal S^*$ such that $w^{-1} w'$ reverses to $v'v^{-1}$ for some $v \in \mathcal S^*$, if such a word exists.
\end{defn}

Dehornoy \cite{Deh2} has proved that the cube condition is equivalent to some expression involving the complement operation:
\begin{prop}
Assume that $(\mathcal S, \mathcal R)$ is a complemented semigroup presentation. Then, for all words $u, u', u'' \in \mathcal S^*$, the following are equivalent:
\begin{enumerate}
\item $(\mathcal S, \mathcal R)$ satisfies the cube condition on $\{u, u', u'' \}$.
\item either $(u \backslash u') \backslash (u \backslash u'')$ and
$(u' \backslash u) \backslash (u' \backslash u'')$ are $\mathcal R$-equivalent or they are not defined,
and the same holds for all permutations of $u, u', u''$.
\end{enumerate}
\end{prop}

Here, we survey some important consequences and applications arising
from the completeness property.

\begin{prop}[\cite{Deh2}, Proposition 6.1]
Every monoid that admits a complete complemented presentation is left-cancellative (i.e., $xy=xz \Rightarrow y=z$).
\end{prop}

\begin{prop}[\cite{Deh2}, Proposition 6.10]
Assume that $(\mathcal S, \mathcal R)$ is a complete semigroup presentation. If $(\mathcal S,\mathcal R)$ is
complemented, then the monoid $\langle \mathcal S | \mathcal R \rangle^+$ admits least common multiples.
\end{prop}

\subsection{Completeness of conjugation-free geometric presentations} \label{result}
In this section, we prove that a conjugation-free geometric presentation is complete if its corresponding graph has no edges:

\begin{prop}
Let $\cL$ be a real arrangement whose fundamental group has a conjugation-free geometric presentation and its graph $G(\cL)$ has no edges. Then the presentation of the corresponding monoid is complete (and complemented).
\end{prop}

\begin{proof}
It is obvious that the conjugation-free geometric presentations are homogeneous (since all the words in the same relation are of the same length).  Hence, we prove this proposition by verifying the equivalent version of the cube condition (for any triple $(u,u',u'') \in ({\mathcal S^*})^3$, the words $(u \backslash u') \backslash (u \backslash u'')$ and $(u' \backslash u) \backslash (u' \backslash u'')$ are $\mathcal R$-equivalent) case-by-case.

\medskip

\begin{enumerate}
\item[Case 1:] The three generators correspond to three lines $\ell_i,\ell_j,\ell_k$ intersecting in three simple points, see Figure \ref{case1}.

\begin{figure}[!ht]
\epsfysize 2cm
\centerline{\epsfbox{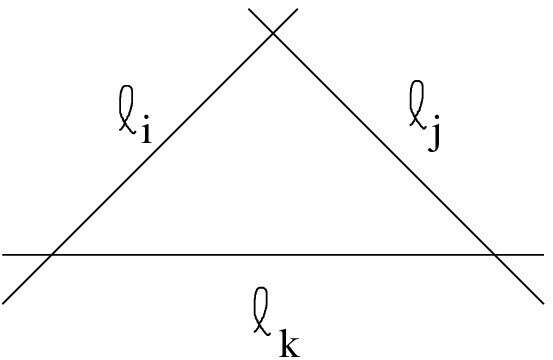}}
\caption{Case 1}\label{case1}
\end{figure}

In this case, the relations induced by the three simple points are: $[x_i,x_j]=[x_i,x_k]=[x_j,x_k]=e$, where $x_i,x_j,x_k$ are the generators of
the lines $\ell_i,\ell_j,\ell_k$ respectively.
So, we have:
$$(x_i \backslash x_j) \backslash (x_i \backslash x_k) = x_k = (x_j \backslash x_i) \backslash (x_j \backslash x_k),$$
which are indeed $\mathcal R$-equivalent.

By symmetry, this holds to any permutation of $x_i,x_j,x_k$ as needed.

\medskip

\item[Case 2:] The three generators correspond to three lines $\ell_i,\ell_j,\ell_k$ passing through the same multiple point, see Figure \ref{case2}. For this case, we have two subcases: in the first case, the corresponding lines appear consecutively in the intersection point. In the second case, the corresponding lines appear separately in the intersection point.

\begin{figure}[!ht]
\epsfysize 3cm
\centerline{\epsfbox{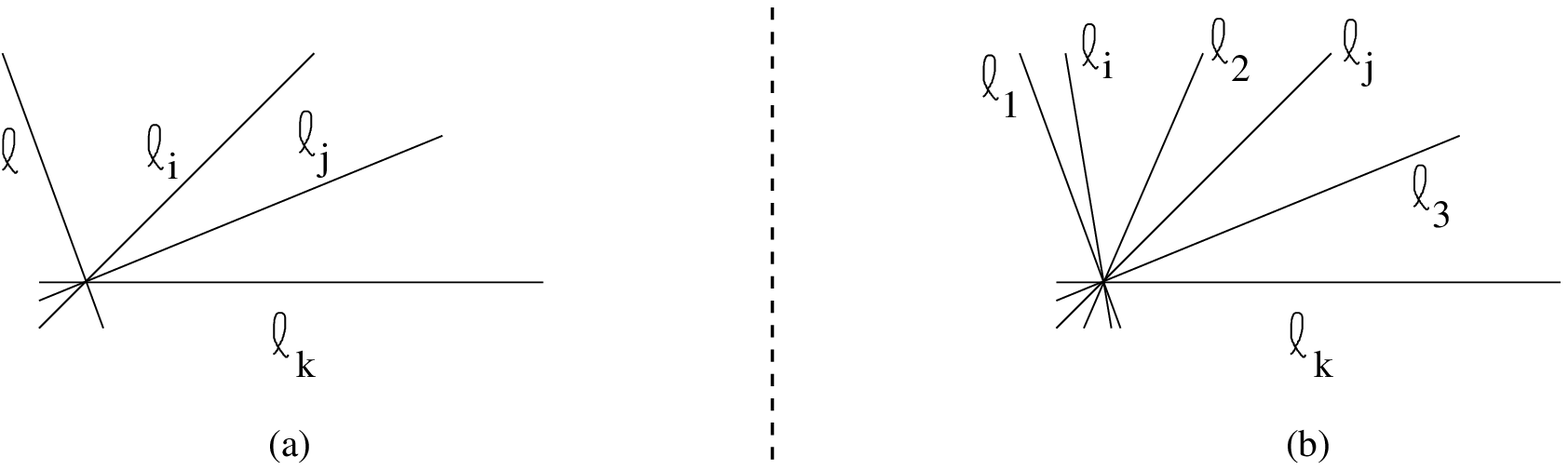}}
\caption{Case 2}\label{case2}
\end{figure}

\medskip

\begin{enumerate}
\item[Case 2a:] {\it The lines appear consecutively in the intersection point:} Without loss of generality, we can assume that the multiple point has multiplicity 4, see Figure \ref{case2}(a). Hence, the relations induced by this multiple point are:
    $$x x_k x_j x_i = x_k x_j x_i x = x_j x_i x x_k = x_i x x_k x_j,$$
    where $x_i,x_j,x_k,x$ are the generators of the lines $\ell_i,\ell_j,\ell_k,\ell$ respectively.
    Hence, we have:
    $$(x_i \backslash x_j) \backslash (x_i \backslash x_k) = e = (x_j \backslash x_i) \backslash (x_j \backslash x_k),$$
    which are indeed $\mathcal R$-equivalent.

    Any other permutation of $x_i,x_j,x_k$ yields $e$ in both sides of the condition, so the condition is satisfied for any permutation.

\medskip

\item[Case 2b:] {\it The lines do not appear consecutively in the intersection point:} Without loss of generality, we can assume that the multiple point has multiplicity 6, see Figure \ref{case2}(b). Hence, the relations induced by this multiple point are:
    \begin{eqnarray*}
    z x_k y x_j x x_i & = & x_k y x_j x x_i z = y x_j x x_i z x_k = x_j x x_i z x_k y =\\
    & =& x x_i z x_k y x_j = x_i z x_k y x_j x,
    \end{eqnarray*}
    where $x_i,x_j,x_k,x,y,z$ are the generators of the lines $\ell_i$,$\ell_j$, $\ell_k$,$\ell_1$,$\ell_2$,$\ell_3$ respectively.
    Hence, we have:
    $$(x_i \backslash x_j) \backslash (x_i \backslash x_k) = e = (x_j \backslash x_i) \backslash (x_j \backslash x_k),$$
    which are indeed $\mathcal R$-equivalent.

    Any other permutation of $x_i,x_j,x_k$ yields $e$ in both sides of the condition, so the condition is satisfied for any permutation.

\end{enumerate}

\end{enumerate}

\medskip

Hence, we have verified the equivalent version of the cube condition ($(u \backslash u') \backslash (u \backslash u'')$ and $(u' \backslash u) \backslash (u' \backslash u'')$ are $\mathcal R$-equivalent) for any triple of generators $u, u', u''$ in any case that the graph has no edges, so we are done.
\end{proof}

Hence, we have the following corollary:
\begin{cor}
Let $\cL$ be a real arrangement whose fundamental group has a conjugation-free geometric presentation and its graph $G(\cL)$ has no edges. Then, the corresponding monoid is cancellative and has least common multiples.
\end{cor}

\begin{rem}
The condition that the graph has no edges is essential, since if we take a line arrangement whose graph contains an edge, and its fundamental group has a conjugation-free geometric presentation, we can find a triple of generators for which the cube condition is not satisfied anymore.
\end{rem}

\section*{Acknowledgments}
We would like to thank an anonymous referee of our previous paper \cite{EGT} for pointing us out the possible connection between the conjugation-free geometric presentations and Dehornoy's complete presentations.
We owe special thanks to Patrick Dehornoy for many discussions.

We wish to thank Tadashi Ishibe for pointing us a wrong consequence from the published version of this paper, and consequently it weakens our result concerning complete presentations.

\end{document}